\documentclass[final]{siamltex}
\usepackage{latexsym,url,amscd,amsmath,amssymb}
\usepackage{graphics}
\usepackage{graphicx}
\usepackage{epstopdf}
\usepackage[usenames]{color}
\usepackage{comment}
\usepackage[misc]{ifsym}
\usepackage[algo2e,linesnumbered,vlined,ruled]{algorithm2e}

\newtheorem{assumption}{Assumption}
\newtheorem{remark}{Remark}

\usepackage{longtable}
\usepackage{color}

\def\inner#1#2{ \langle #1, #2 \rangle}
\def\diag#1{ \text{diag}(#1) }

\begin{document}
\title{Applications of gauge duality in robust principal component analysis and semidefinite programming}

\author{Shiqian Ma\footnotemark[1]   \and  Junfeng Yang\footnotemark[2]}
\renewcommand{\thefootnote}{\fnsymbol{footnote}}

\footnotetext[1]{Department of Systems Engineering and Engineering Management, The Chinese University of Hong Kong, Hong Kong (Email: {\tt sqma@se.cuhk.edu.hk}).
This author was supported by Hong Kong Research Grants Council General Research Fund (Project ID: 14205314).
}

\footnotetext[2]{Department of Mathematics, Nanjing University, China  (Email: {\tt jfyang@nju.edu.cn}). This author was supported by the Natural Science Foundation of China (NSFC-11371192). This work was done while this author was visiting the Chinese University of Hong Kong.
}

\renewcommand{\thefootnote}{\arabic{footnote}}
\date{\today}
\maketitle

\begin{abstract}
Gauge duality theory was originated by Freund [Math. Programming, 38(1):47-67, 1987] and was recently further investigated by Friedlander, Mac{\^e}do and Pong [SIAM J. Optm., 24(4):1999-2022, 2014]. When solving some matrix optimization problems via gauge dual, one is usually able to avoid full matrix decompositions such as singular value and/or eigenvalue decompositions. In such an approach, a gauge dual problem is solved in the first stage, and then an optimal solution to the primal problem can be recovered from the dual optimal solution obtained in the first stage. Recently, this theory has been applied to a class of \emph{semidefinite programming} (SDP) problems with promising numerical results [Friedlander and Mac{\^e}do, SIAM J. Sci. Comp., to appear, 2016].
In this paper, we establish some theoretical results on applying the gauge duality theory to robust  \emph{principal component analysis} (PCA) and general SDP. For each problem, we present its gauge dual problem, characterize the optimality conditions for the primal-dual gauge pair, and validate a way to recover a primal optimal solution from a dual one. These results are extensions of [Friedlander and Mac{\^e}do, SIAM J. Sci. Comp., to appear, 2016] from nuclear norm regularization to robust PCA and from a special class of SDP which requires the coefficient matrix in the linear objective to be positive definite to SDP problems without this restriction. Our results provide further understanding in the potential advantages and disadvantages of the gauge duality theory.
\end{abstract}



\thispagestyle{plain}

\section{Introduction}
Nowadays, optimization problems with matrix decision  variables, smooth or nonsmooth, convex or nonconvex,  including  low-rank matrix approximation and completion, robust \emph{principal component analysis} (PCA) and \emph{semidefinite programming} (SDP) problems, just to name a few, are prevalent due to their tremendous applications in the  era of big data.
Most of the classical  approaches for solving these problems are iterative and require full matrix decompositions at each iteration, such as \emph{singular value decomposition} (SVD) and eigenvalue decomposition, etc., which are costly or even impossible for large scale computations.
As a consequence, when the size of the matrix variable is becoming   larger and larger, as is the case in many massive data analysis tasks, classical approaches inevitably meet bottlenecks caused by the computation of full matrix decompositions.
To overcome this difficulty, efficient and robust algorithms which do not rely on full matrix decompositions are undoubtedly attractive.

In this paper, we focus on two specific problems, i.e., robust PCA and SDP, and present for each problem a gauge duality approach, which  in principle is  able to avoid full matrix decompositions.
Gauge duality theory was originally developed in \cite{Fre87} and was recently further studied under a more concrete setting in \cite{FMP14}.
Lately, it has  been successfully applied in \cite{FM15} to solving a form of trace minimization problem that arises naturally from the lifting reformulations of phase retrieval and blind deconvolution problems. In particular, the trace minimization problem considered in \cite{FM15} contains a special class of SDP problems. The numerical results reported there for phase retrieval and blind deconvolution   are favorable.
Roughly speaking, gauge duality approach is a two-stage method. In the first stage, a gauge dual problem is solved via any algorithms relying merely on subgradients, e.g., subgradient and bundle type methods \cite{Sho85,LNN95,Kiw95,BN05}; and, in the second stage, an optimal solution to the primal problem is recovered from the dual optimal one obtained in the first stage.
Since the computation of a subgradient of the gauge dual objective function does not rely on full matrix decomposition and, furthermore, the projection on the gauge dual constraints is very simple, this approach is potentially advantageous compared to the classical Lagrange dual approach.
We emphasize that our main focus in this paper is on the theoretical side, that is to establish gauge duality theory for both robust PCA and SDP, while the   realization  of the established theory  deserves further investigations.

The rest of this paper is organized as follows. In Section \ref{sec:gauge-duality}, we briefly introduce the gauge duality theory, including
the original work \cite{Fre87} and the recent advances \cite{FMP14,FM15}. In Sections \ref{sec:rpca} and \ref{sec:sdp}, we study gauge duality   for robust PCA and SDP, respectively. For each problem, we derive its gauge dual problem, characterize the optimality conditions, and validate a way to recover a primal optimal solution from a dual one. Finally, we give some concluding remarks in Section \ref{sec:concluding}.

\section{Review of gauge duality}\label{sec:gauge-duality}
In this section, we review  briefly the gauge duality theory \cite{Fre87,FMP14} and the recent applications in low-rank optimization \cite{FM15}. Our contributions are also summarized in this section. We will see in Sections \ref{sec:rpca} and \ref{sec:sdp} that both robust PCA and SDP  can be viewed, after reformulation, as gauge optimization problems, and methods based on gauge duality can thus be designed.

\subsection{Gauge duality: the work of Freund \cite{Fre87}}\label{subsc:Freund}
In the original work \cite{Fre87}, Freund defined gauge optimization as minimizing a closed gauge function over a closed convex set.
An extended real valued function $\kappa: \Re^n\rightarrow \Re \cup \{+\infty\}$ is called a gauge function if it is nonnegative, convex, positively homogeneous,  and vanishes at the origin.
In particular, norms   are symmetric gauges which are finite everywhere.
According to \cite{Fre87}, a general nonlinear gauge optimization problem is to minimize a closed gauge function $\kappa$ over a closed convex set $X \subset \Re^n$, i.e.,
\begin{equation}\label{NG}
  \min\nolimits_x \{ \kappa(x) \ | \  x \in X\}.
\end{equation}
A point $x$ is said to be feasible for \eqref{NG} if $x\in X$, otherwise it is infeasible. If $x\in X$ and $\kappa(x) = +\infty$, then $x$ is essentially infeasible.
If $x\in X$ and $\kappa(x) < +\infty$, then $x$ is said to be strongly  feasible.
In order to define the gauge dual problem of \eqref{NG}, we need the notions of polar function of a gauge and the antipolar set of a given closed convex set.
The polar function of a gauge  $\kappa$ is defined by
\begin{equation}\label{polar-fun}
  \kappa^\circ(y) = \inf\{\mu > 0 \ | \inner{x}{y} \leq \mu  \kappa(x) \text{ for all } x \},
\end{equation}
see, e.g., \cite{Roc70}. According to this definition,  it is apparent that $\inner{x}{y} \leq \kappa(x)\kappa^\circ(y)$ for any
$x \in \text{dom}(\kappa)$ and $y\in \text{dom}(\kappa^\circ)$, where $\text{dom}(\cdot)$ denotes the effective domain of a function.
An equivalent definition of the polar function of a gauge $\kappa$ is given by
  $\kappa^\circ(y) = \sup\{\inner{x}{y} \ | \ \kappa(x) \leq 1 \}$.
Thus, the polar function of any given norm is simply the corresponding dual norm.
The antipolar set $X'$ of a closed convex set $X$ is given by
$X' = \{y\in\Re^n \ | \ \inner{x}{y} \geq 1 \text{ for all } x \in X\}$.
%
%
According to \cite{Fre87}, the nonlinear gauge dual problem of \eqref{NG} is defined as  minimizing the  polar function $\kappa^\circ$ over the antipolar set $X'$, i.e.,
\begin{equation}\label{NGD}
  \min\nolimits_y \{\kappa^\circ(y) \ | \  y \in X'\}.
\end{equation}
It is easy to show from \eqref{polar-fun} that the polar function of a gauge function is again a gauge function. Thus, the gauge dual problem \eqref{NGD} is also a gauge optimization problem.
A set $X$  is called a ray-like set if  $x + \theta y\in X$ for all $x, y\in X$ and $\theta\geq 0$.
It can be shown that the antipolar $X'$ of any set $X$ is also  ray-like, and if $X$ is closed, convex, ray-like, and does not contain the origin, then $X'' = X$.
Consequently,  the gauge dual of \eqref{NGD} is precisely \eqref{NG} whenever $\kappa$ is closed and $X$ is closed, convex, ray-like and does not contain the origin.

The following theorem summarizes the main theoretical results presented in \cite{Fre87} for the primal-dual gauge pair \eqref{NG} and \eqref{NGD}. In particular, the strong duality results \cite[Theorem 2A]{Fre87} require the constraint set to satisfy the  ray-like condition described above.

\begin{theorem}[Freund's gauge duality, {\cite[Theorem 1A, Theorem 2A]{Fre87}}]\label{Freund-thm}
Let $p^*$ and $d^*$ be the optimal objective value  of \eqref{NG} and \eqref{NGD}, respectively.
Denote the relative interior of a set by $\mathrm{ri}(\cdot)$.
\begin{enumerate}
  \item If $x$ and $y$ are strongly feasible for  \eqref{NG} and \eqref{NGD}, respectively, then $\kappa(x)\kappa^\circ(y)\geq 1$. Hence, $p^*d^*\geq 1$.
  \item If $p^* = 0$, then \eqref{NGD} is essentially infeasible, i.e., $d^*=+\infty$.
  \item If $d^* = 0$, then \eqref{NG} is essentially infeasible, i.e., $p^*=+\infty$.
  \item If $\bar x$ and $\bar y$ are strongly feasible for  \eqref{NG} and \eqref{NGD}, respectively, and $\kappa(\bar x)\kappa^\circ(\bar y) = 1$, then they are respectively optimal for \eqref{NG} and \eqref{NGD}.
  \item Assume that $\kappa$ is closed and $X$ is closed, convex, and ray-like.
  If  $\mathrm{ri}(X) \cap   \mathrm{ri}(\mathrm{dom}(\kappa)) \neq \emptyset$ and $\mathrm{ri}(X') \cap   \mathrm{ri}(\mathrm{dom}(\kappa^\circ)) \neq \emptyset$,
  then $p^*d^*=1$, and each problem attains its optimum value.
\end{enumerate}
\end{theorem}
The inequality $p^*d^*\geq 1$ is called weak gauge duality, which holds for the primal-dual gauge pair as long as both $p^*$ and $d^*$ are positive and finite. On the other hand, the equality $p^*d^* = 1$ is called strong gauge duality. The sufficient conditions given in part 5 of Theorem \ref{Freund-thm} to guarantee strong duality and attainments of the primal and the dual optimum values are relatively easy to verify and are different from   constraint qualifications commonly studied for Lagrange dual. This theory was applied in \cite{Fre87} to linear programming, convex quadratic programming, and $\ell_p$-norm minimization problems with linear constraints.

\subsection{Gauge duality: the work of Friedlander, Mac{\^e}do and Pong \cite{FMP14}}
Very recently, Friedlander, Mac{\^e}do and Pong further studied gauge duality theory in \cite{FMP14}
under a more concrete setting, that is
\begin{eqnarray}\label{Fried-GP}
  \min\nolimits_x \{ \kappa(x) \ | \ \rho(Ax - b) \leq \varepsilon \}.
\end{eqnarray}
Here $\kappa$ and $\rho$ are closed gauge functions, and $A\in\Re^{m\times n}$, $b\in\Re^m$ and $\epsilon \geq 0$ are given.
The formulation \eqref{Fried-GP} is of particular interest because it encompasses many important data analysis applications in which  the constraint set $\{x \ | \ \rho(Ax - b) \leq \varepsilon \}$ represents data fidelity and the objective function $\kappa(x)$ represents certain regularization.
It is easy to show, see \cite[Sec. 4.2]{FMP14}, that the Lagrange dual problem of \eqref{Fried-GP} is given by
\begin{eqnarray}\label{Fried-LD}
\max\nolimits_y \{   b^\top  y - \varepsilon \rho^\circ(y) \ | \   \kappa^\circ(A^\top y)  \leq 1\}.
\end{eqnarray}
To derive the gauge dual problem, one needs to express the antipolar set of $\{x \ | \ \rho(Ax - b) \leq \varepsilon \}$ explicitly in terms of the given function and data.  For this purpose, a set of antipolar calculus rules were derived in \cite[Sec. 3]{FMP14} based on basic convex analysis results in \cite{Roc70}. In particular, these rules lead to a neat calculation of the antipolar set of $\{x \ | \ \rho(Ax - b) \leq \varepsilon \}$.
The gauge dual problem of \eqref{Fried-GD} is given by
\begin{eqnarray}\label{Fried-GD}
   \min\nolimits_y \{  \kappa^\circ(A^\top y) \ | \  b^\top  y - \varepsilon \rho^\circ(y)  \geq 1\}.
\end{eqnarray}
By comparing  the two  dual problems, one can see that the functions $b^\top  y - \varepsilon \rho^\circ(y) $ and $\kappa^\circ(A^\top y)$ switched their positions, and ``$\leq$" in \eqref{Fried-LD} is replaced by ``$\geq$" in  \eqref{Fried-GD}.
This could hint that gauge dual problem is preferred compared to Lagrange dual
when this switching simplifies the constraints, e.g., when the projection on the set $\{y \ | \ b^\top  y - \varepsilon \rho^\circ(y)  \geq 1\}$ is much easier than on $\{y \ | \ \kappa^\circ(A^\top y)  \leq 1\}$.
We refer to \cite{FMP14} for details on the derivation of \eqref{Fried-GD}.
Besides, some variational properties of the  optimal value of \eqref{Fried-GP} are also analyzed in \cite[Sec. 6]{FMP14} in terms of perturbations in $b$ and $\varepsilon$.

As mentioned in Section \ref{subsc:Freund}, Freund's strong gauge duality  results \cite[Theorem 2A]{Fre87} require the ray-like property of the constraint set. By connecting gauge dual problem \eqref{NGD} with the Fenchel dual problem of \eqref{NG}, this ray-like requirement on $X$ was removed in \cite{FMP14} and strong gauge duality still holds between \eqref{NG} and \eqref{NGD} as long as certain conditions analogous to those required by Lagrange duality are satisfied. Roughly speaking, this connection says that, under regularity conditions, gauge dual problem \eqref{NGD} is equivalent to the Fenchel dual problem of \eqref{NG} in the sense that solving one of them will solve the other, see \cite[Theorem 5.1]{FMP14}. The strong duality results  given in \cite[Corollary 5.2, Corollary 5.4]{FMP14}  were derived based on this connection and are summarized below.

\begin{theorem}[Strong duality, {\cite[Corollary 5.2, Corollary 5.4]{FMP14}}]\label{FMP-thm}
Let $p^*$ and $d^*$ be the optimal objective value  of \eqref{NG} and \eqref{NGD}, respectively.
\begin{enumerate}
\item If  $\mathrm{dom} (\kappa^\circ) \cap   X' \neq \emptyset$ and $\mathrm{ri} (\mathrm{dom}(\kappa)) \cap   \mathrm{ri}(X) \neq \emptyset$,
  then $p^*d^*=1$  and the gauge dual problem \eqref{NGD}  attains its optimum value.

\item If $\kappa$ is closed, and that  $\mathrm{ri} (\mathrm{dom}(\kappa)) \cap   \mathrm{ri}(X) \neq \emptyset$ and $\mathrm{ri}(\mathrm{dom} (\kappa^\circ)) \cap \mathrm{ri}(X') \neq \emptyset$, then $p^*d^*=1$  and both \eqref{NG} and \eqref{NGD}  attain their optimum values.
\end{enumerate}
\end{theorem}

Compared to the strong duality results given in \cite{Fre87}, here the ray-like condition on $X$ is no longer needed.
Besides theoretical duality results and the derivation of \eqref{Fried-GD}, the gauge dual of optimization problems with composition and conic side constraints were also discussed in \cite[Sec. 7]{FMP14}. In particular, the gauge dual problem of general SDP was derived in \cite[Sec. 7.2.1]{FMP14}.

\subsection{Gauge duality theory applied in low-rank optimization \cite{FM15}}
We now review briefly the recent application of gauge duality theory in low-rank optimization \cite{FM15}. Let ${\cal H}^n$ be the set of Hermitian matrices of order $n$. The model problem  studied in \cite{FM15} is
\begin{eqnarray}\label{Fried-trace}
  \min_{X\in{\cal H}^n} \{\langle I, X\rangle \ | \ \|{\cal A} X - b\| \leq \varepsilon, \ X \succeq 0\},
\end{eqnarray}
where  $I$ denotes the identity matrix of order $n$, $X\succeq 0$ means that $X$ is Hermitian and positive semidefinite, and the constraint
$\|{\cal A} X - b\| \leq \varepsilon$ enforces data fitting of the noisy linear measurements $b$.
By using the indicator function of the set $\{ X\in {\cal H}^n \ | \ X\succeq 0\}$,   \eqref{Fried-trace} can be easily reformulated as a gauge optimization problem, whose gauge dual problem is given by
\begin{eqnarray}\label{Fried-trace-gd}
  \min\nolimits_{y} \{\lambda_{\max}({\cal A}^\top y) \ | \ b^Ty - \varepsilon \|y\|_* \geq 1\}.
\end{eqnarray}
Here ${\cal A}^\top $ denotes the adjoint operator of $\cal A$, $\lambda_{\max}(\cdot)$ denotes the maximum eigenvalue, and $\|\cdot\|$ represents the dual norm of $\|\cdot\|_*$.
Strong duality result between \eqref{Fried-trace} and \eqref{Fried-trace-gd} was established and optimality conditions were derived in \cite{FM15}.
Furthermore, based on the established optimality conditions, a primal solution recovery procedure was validated. For the case of $\ell_2$-norm, the whole gauge dual based framework is  first to solve the gauge dual problem \eqref{Fried-trace-gd}  and then to recover a solution of \eqref{Fried-trace} via solving a reduced linear least squares problem. The projection onto the constraint set of \eqref{Fried-trace-gd} is easy, and the computation of a subgradient of the objective only requires computing an eigenvector corresponds to the largest eigenvalue of ${\cal A}^\top y$ for a given $y$. These are the main advantages of considering the gauge dual problem \eqref{Fried-trace-gd}. In contrast, the projection onto the set $\{y \ | \ \lambda_{\max}({\cal A}^\top y) \leq 1\}$, which appear in the Lagrange dual of \eqref{Fried-trace} as a constraint, can be very costly.
In \cite{FM15}, this gauge dual based algorithmic framework was applied to solve the phase retrieval problem and the blind deconvolution problem, both of which can be cast into \eqref{Fried-trace} by using the phase-lift reformulation introduced in \cite{CSV13}. A projected spectral subgradient algorithm with nonmonotonic line search was applied to solve \eqref{Fried-trace-gd} in \cite{FM15}. An important reason for the promising experimental results in \cite{FM15} is that in these applications the phase lifted problems admit rank-$1$ solutions, which makes the objective function of \eqref{Fried-trace-gd} essentially differentiable when the iterates are close to a true solution.
In \cite[Sec. 6]{FM15}, it was also discussed that the gauge dual framework can be applied to solve a generalization of \eqref{Fried-trace}, that is
\begin{eqnarray}\label{Fried-Cpd}
  \min_{X\in{\cal H}^n} \{\langle C, X\rangle \ | \ \|{\cal A} X - b\| \leq \varepsilon, \ X \succeq 0\},
\end{eqnarray}
where $C$ is required to be positive definite. In particular, the authors of \cite{FM15} reduced \eqref{Fried-Cpd} to \eqref{Fried-trace} by using $C^{-\frac{1}{2}}$. A primal solution recovery procedure thus follows from the result for \eqref{Fried-trace}. In Section \ref{sec:sdp}, we will remove the positive definite restriction on $C$.

\subsection{Our contributions}
Our main contributions are summarized below.  In Section \ref{sec:rpca}, we specialize this gauge dual framework to robust PCA. There the primal solution recovery procedure becomes a reduced $\ell_1$-norm minimization problem, rather than a least squares problem as in \cite{FM15}.
By setting $\varepsilon=0$,  \eqref{Fried-Cpd} reduces to a special class of SDP since $C$ must be positive definite. In Section \ref{sec:sdp}, we discuss the gauge dual framework for SDP without requiring $C$ to be positive definite. Particularly, we give a sufficient condition to guarantee strong gauge duality and derive a primal solution recovery procedure based on optimality conditions.
These results are extensions of \cite[Corollary 6.1, Corollary 6.2]{FM15}.

\section{Application to robust PCA}\label{sec:rpca}
Given a  matrix $M\in\Re^{m\times n}$. We consider the problem of decomposing $M$ into the sum of a low-rank component $X$ and a sparse component $Y$, which is known as the robust PCA problem \cite{CSPW11,CLMW11} and has many applications in image and signal processing and data analysis. After recovering $X$,
the principal components can then be obtained via applying a SVD  to $X$.
In the pioneering work \cite{CSPW11,CLMW11}, the authors proposed to recover the low-rank and the sparse components of $M$ via solving the following convex model
\begin{equation}\label{RPCA}
\min\nolimits_{X,Y} \{\|X\|_* + \gamma \|Y\|_1  \ | \ X + Y = M\}.
\end{equation}
Here $\|X\|_*$ denotes the nuclear norm (the sum of all singular values) of $X$,   $\|Y\|_1$ represents the sum of the absolute values of all entries of $Y$,
and $\gamma>0$ balances the two terms for minimization. Particularly,  $\|X\|_*$ and $\|Y\|_1$ serve as  surrogate convex functions of the rank of $X$ and the number of nonzero entries of $Y$, respectively. We refer to  \cite{CSPW11,CLMW11} for recoverability results of \eqref{RPCA}.
Throughout this section, we assume that (i) $M$ is nonzero, and (ii) the trivial decompositions $(M,0)$ and $(0,M)$ are not  optimal solutions  of \eqref{RPCA}.

There exist a few other convex formulations for the robust PCA problem. For example,
\begin{eqnarray}
\label{reform-1}
\min\nolimits_{X,Y}  \{ \|X + Y - M\|_F^2 \ | \  \|X\|_* \leq \tau_L,  \|Y\|_1  \leq \tau_S \}
\end{eqnarray}
and
\begin{eqnarray}
\label{reform-2}
\min\nolimits_{X,Y} \{  \gamma_L \|X\|_* + \gamma_S \|Y\|_1  +  \|X + Y - M\|_F^2\},   
\end{eqnarray}
where $\|\cdot\|_F$ denotes the Frobenius norm, and $\tau_L, \tau_S, \gamma_L, \gamma_S$ are positive parameters.
First order methods for solving these models, such as projected gradient method, proximal gradient method,
and alternating direction method of multipliers (ADMM), usually require either projections onto nuclear norm ball  or shrinkage/thresholding on the singular values of certain matrices, and thus full SVD is unavoidable. Indeed, the classical Frank-Wolfe method \cite{FW56}, also known as conditional gradient method \cite{LP66}, can be applied to solve appropriate convex formulations of robust PCA in which full SVD can be avoided, see, e.g., \cite{MZWG14}. However, Frank-Wolfe type methods usually suffer from slow sublinear convergence, and, moreover, the steplength parameter is hard to choose.
On the other hand, nonconvex optimization formulations are also useful to avoid full SVD. For example,  the authors of \cite{SWZ14} proposed to solve
the  following nonconvex optimization problem via  ADMM
\[\min\nolimits_{U,V,Y} \{\|Y\|_1 \ | \ UV^\top  + Y = M, U\in \Re^{m\times r}, V\in \Re^{n\times r}, Y \in \Re^{m\times n} \}.\]
Indeed, full SVD is replaced by   QR factorization of a small $m\times r$  matrix at each step. The potential drawback of this approach is that the rank $r$ is mostly unknown and could be difficult to estimate in practice.
In short, each approach has its own advantages and disadvantages, and the gauge dual approach to be studied is another candidate method   to avoid full SVD or full eigenvalue decomposition. In this paper, we only focus on the theory, while the  practical realizations of the established theory need to be investigated further. 

\subsection{The gauge dual problem of \eqref{RPCA}}
  Let $Z\in \Re^{m\times n}$. We denote the spectral norm and the max norm of $Z$ by  $\|Z\|_2$ and $\|Z\|_\infty$, respectively, i.e.,
\[
\|Z\|_2 = \sqrt{\lambda_{\max}(Z^\top Z)}  \text{~~and~~} \|Z\|_\infty = \max \{ |Z_{ij}| \ | \ 1\leq i \leq m, 1\leq j \leq n\},
\]
where $\lambda_{\max}(\cdot)$ denotes the  largest eigenvalue. It is easy to derive  from standard Lagrange duality theory that
the Lagrange dual problem of \eqref{RPCA} is
\begin{equation}\label{RPCA-LD}
  \max\nolimits_Z \{ \langle M, Z\rangle \ | \ \max \{ \|Z\|_2, \|Z\|_\infty/\gamma\} \leq 1\}.
\end{equation}
Although the projection onto each single constraint set of \eqref{RPCA-LD} (either $\|Z\|_2\leq 1$ or $\|Z\|_\infty \leq \gamma$) has a closed form representation,
that onto their intersection seems to be unreachable to the best of our knowledge.
We note that  the projection onto the  set $\{ Z \ | \ \|Z\|_2\leq 1\}$   needs a full SVD.
 Define
 \[
 {\cal A}(X,Y) := X + Y \text{~~and~~} \kappa(X,Y) := \|X\|_* + \gamma \|Y\|_1.
 \]
 Clearly, $\cal A$ is a linear operator and $\kappa$ is a closed gauge function.
It is thus apparent that \eqref{RPCA} can be rewritten as the following gauge optimization problem
 \begin{eqnarray}\label{RPCA-gauge-form}
   \min\nolimits_{X,Y} \{ \kappa(X,Y) \ | \ {\cal A}(X,Y) - M = 0 \}.
 \end{eqnarray}
Let $\cal A^\top$ be the adjoint operator of $\cal A$, then ${\cal A}^\top Z = (Z, Z)$.  According to \cite[Prop. 2.4]{FMP14}, it holds that
\begin{eqnarray*}
   \kappa^\circ(U, V) &=&  \max \{ \|U\|_2, \|V\|_\infty/\gamma\},  \\
  \kappa^\circ({\cal A}^\top Z) &=& \kappa^\circ(Z,Z) = \max \{ \|Z\|_2, \|Z\|_\infty/\gamma\}.
\end{eqnarray*}
Furthermore, it follows from \cite{FMP14} (or   \eqref{Fried-GD} with $\rho$ being any matrix norm and $\varepsilon = 0$) that the gauge dual  of \eqref{RPCA-gauge-form} is  $\min\nolimits_Z \{ \kappa^\circ({\cal A}^\top Z) \ | \  \langle M, Z\rangle \geq 1\}$, or equivalently,
 \begin{eqnarray}\label{RPCA-GD}
   \min\nolimits_Z \{ \max \{ \|Z\|_2, \|Z\|_\infty/\gamma\} \ | \ \langle M, Z\rangle \geq 1\}.
 \end{eqnarray}
%

\subsection{Strong duality and optimality conditions}
Next, we derive strong duality and optimality conditions for the primal-dual gauge pair \eqref{RPCA} and \eqref{RPCA-GD}.
For a real matrix $T$, we let $\sigma(T)$ denote the vector formed by stagnating the singular values of  $T$ in a nonincreasing order.
The set of orthonormal matrices of order $n$ is denoted by ${\cal U}_n$.
Our analysis relies on the well-known von Neumann's trace inequality and a basic lemma from convex analysis, which we present below.

\begin{theorem}
  [von Neumann]\label{thm:vonNeumann}
For any $m\times n$ real matrices $X$ and $Y$, it holds that
\[
\inner{X}{Y} \leq \inner{\sigma(X)}{\sigma(Y)}.
\]
Furthermore, equality holds if and only if there is a simultaneous SVD of $X$ and $Y$, i.e., for  some $U\in{\cal U}_m$ and $V\in{\cal U}_n$, there hold
$X = U \diag{\sigma(X)} V^\top $ and $Y = U \diag{\sigma(Y)} V^\top $.
\end{theorem}
\begin{lemma}\label{thm:dual-norm}
  Let $\|\cdot\|_\diamond$ be any norm and $\|\cdot\|_\star$ be its dual norm, i.e.,
$\|y\|_\star = \sup \{ \langle x, y\rangle \ | \ \|x\|_\diamond \leq 1\}$.
Then, for any $x$ and $y$, $x\in\partial \|y\|_\star$ if and only if  $\|x\|_\diamond \leq 1$  and
$\inner{x}{y} = \|y\|_\star$. Furthermore, if $y\neq 0$, then $\|x\|_\diamond = 1$.
\end{lemma}

Denote the feasible set of \eqref{RPCA-gauge-form} by $\cal C$, i.e., ${\cal C} := \{(X,Y) \ | \ X + Y = M \}$. It is easy to show that the ray-like condition fails to hold for ${\cal C}$ and thus the results given in \cite{Fre87} cannot be used to establish strong duality. Instead, we resort to the results obtained in \cite{FMP14}, in which the ray-like condition was removed. Direct calculation  shows that the antipolar set
${\cal C}'$ of $\cal C$
is given by
$  {\cal C}' = \{(U,V) \ | \ U = V, \; \langle M, V\rangle \geq 1\}$.
Thus, the primal-dual gauge pair \eqref{RPCA-gauge-form} and \eqref{RPCA-GD} can also be rewritten, respectively,  as
\[
\min_{X,Y} \{\kappa(X,Y) \ | \ (X,Y) \in {\cal C}\} \text{~~and~~}
\min_{U,V} \{\kappa^\circ(U,V) \ | \ (U,V) \in {\cal C}'\}.
\]
Clearly, $\kappa$ is a closed gauge function and both $\kappa$ and $\kappa^\circ$ has full domain. It is then apparent that, as long as $M\neq 0$, the conditions of  \cite[Corollary 5.4]{FMP14} are satisfied since
$\mathrm{ri}({\cal C}) \cap   \mathrm{ri}(\mathrm{dom}(\kappa)) \neq \emptyset$ and $\mathrm{ri}({\cal C}') \cap   \mathrm{ri}(\mathrm{dom}(\kappa^\circ)) \neq \emptyset$, and consequently both  \eqref{RPCA-gauge-form} and \eqref{RPCA-GD} achieve their optimal values and strong duality holds.
These results are summarized in the following theorem, which follows directly from   \cite[Corollary 5.4]{FMP14}.

\begin{theorem}[Strong duality]\label{thm:strong-duality}
Assume that $M\neq 0$. Then, both   \eqref{RPCA-gauge-form} and \eqref{RPCA-GD} achieve their optimal values. Furthermore, strong duality holds, i.e.,
\[
\big[\min_{X,Y} \{\|X\|_* + \gamma \|Y\|_1, \  \ X + Y = M\} \big] \cdot
\big[ \min_Z \{ \max \{ \|Z\|_2, \|Z\|_\infty/\gamma\}, \  \ \langle M, Z\rangle \geq 1\} \big]
= 1.
\]
\end{theorem}


%

We next establish optimality conditions for the primal-dual gauge pair \eqref{RPCA} and \eqref{RPCA-GD}, which are critical in designing a procedure to recover a primal optimal solution from a dual one.
\begin{theorem}[Optimality conditions]\label{thm:optimality-conditions}
Assume $M\neq 0$. If $(X,Y)$ and $Z$ satisfy
\begin{enumerate}
  \item $X+Y = M$,
  \item $\langle M, Z \rangle = 1$,
  \item $\|Y\|_1 \|Z\|_\infty = \langle Y, Z\rangle$, 
  \item $\sigma_i(X) (\|Z\|_2 - \sigma_i(Z)) = 0$,  $\forall \  i = 1,2,\ldots,\min(m,n)$,
  \item $X$ and $Z$ admit  a simultaneous ordered SVD,
  \item $\|Z\|_2 = \|Z\|_{\infty}/\gamma$,
\end{enumerate}
then, they are, respectively, optimal for the primal-dual gauge pair \eqref{RPCA} and \eqref{RPCA-GD}. Conversely, if $(X,Y)$ and $Z$
are, respectively, optimal for  \eqref{RPCA} and \eqref{RPCA-GD},  and  neither $X$ nor $Y$ is zero, then these conditions are also necessary.
\end{theorem}
\begin{proof}
Let $e$ be  the vector of all ones.
First, we assume that $(X,Y)$ and $Z$ satisfy conditions 1-6. We need to prove that $(X,Y)$ and $Z$  are primal-dual feasible and satisfy   strong duality. Primal and dual feasibility follows from conditions 1-2. Furthermore, it holds that
\begin{eqnarray*}
 & & (\|X\|_* + \gamma \|Y\|_1) \max \{ \|Z\|_2, \|Z\|_\infty/\gamma\} \\
    &=&  \langle  \sigma(X), e\rangle \|Z\|_2 +  \|Y\|_1   \|Z\|_\infty   \\
    &=&  \langle \sigma(X), \sigma(Z)\rangle  +  \langle Y, Z\rangle   \\
    &=&  \langle X, Z\rangle  +  \langle Y, Z\rangle   \\
    &=&  \langle M, Z\rangle  \\
    &=& 1.
\end{eqnarray*}
Here  the first ``$=$" follows from condition 6, the second ``$=$" follows from conditions 3-4, the third ``$=$" follows from condition 5 and  Theorem \ref{thm:vonNeumann}. 
As a result, strong duality holds.

Conversely, we assume that  $(X,Y)$ and $Z$
are optimal for the gauge dual pair \eqref{RPCA} and \eqref{RPCA-GD}  and  neither $X$ nor $Y$ is zero.
We need to prove that  conditions 1-6 are also necessary.
First, condition 1 is the  primal feasibility and thus is true.
%
By strong duality, we have
\begin{eqnarray*}
  1 &=& (\|X\|_* + \gamma \|Y\|_1) \max \{ \|Z\|_2, \|Z\|_\infty/\gamma\} \\
    &\geq&  \langle  \sigma(X), e\rangle \|Z\|_2 +  \|Y\|_1   \|Z\|_\infty   \\
    &\geq&  \langle \sigma(X), \sigma(Z)\rangle  +  \langle Y, Z\rangle   \\
    &\geq&  \langle X, Z\rangle  +  \langle Y, Z\rangle   \\
    &=&  \langle M, Z\rangle  \\
    &\geq& 1.
\end{eqnarray*}
Here the first ``$=$" is due to strong duality, the first ``$\geq$" is clearly true, the second ``$\geq$" is because $\sigma(X)\geq 0$ and $ \|Y\|_1   \|Z\|_\infty  \geq
 \langle Y, Z\rangle $, the third ``$\geq$"  follows from Theorem \ref{thm:vonNeumann}, the subsequent equality is because $(X,Y)$ is primal feasible, the last ``$\geq$" is because $Z$ is dual feasible.
Thus all of above inequalities become equalities. This proves conditions 2-4.
Condition 5 follows again from  Theorem \ref{thm:vonNeumann}, and finally condition 6 is   true because neither $X$ nor $Y$ is zero.
\end{proof}

\begin{remark}
Let $(X,Y)$ and $Z$
be, respectively, optimal for  \eqref{RPCA} and \eqref{RPCA-GD},  and  neither $X$ nor $Y$ is zero.
Define
${\cal I} = \{(i,j) \ | \ |Z_{ij}| = \|Z\|_\infty\}$ and ${\cal J} = \{(i,j) \ | \ (i,j)\notin {\cal I}, i=1,\ldots,m, j=1,\ldots,n\}$.
Then, the condition   $\|Y\|_1 \|Z\|_\infty = \langle Y, Z\rangle$ implies that (i)
$Y_{ij} = 0$ for all $(i,j) \in {\cal J}$, and (ii) either $Y_{ij} = 0$ or $\text{sgn}(Y_{ij}) = \text{sgn}(Z_{ij})$ for all $(i,j) \in {\cal I}$.
\end{remark}

It can be seen from the proof of Theorem \ref{thm:optimality-conditions} that conditions 1-5 are always necessary for $(X,Y)$ and $Z$ to be primal-dual optimal, and condition 6 is necessary only when neither $X$ nor $Y$ is zero. Given a dual optimal solution $Z$, in practice one can always check whether $(M,0)$ or $(0,M)$ is optimal by examining the strong duality condition. These are not practically interesting cases.

\subsection{Primal solution recovery}
The following theorem validates a way to recover a nontrivial primal optimal solution
$(X,Y)$ from an arbitrary  dual optimal solution $Z$.

\begin{theorem}[Primal solution recovery]\label{thm:primal-recovery}
Assume $M\neq 0$. Let $Z$ be an arbitrary optimal solution for \eqref{RPCA-GD}, $r \in\{1,2, \ldots, \min(m,n)\}$ be the multiplicity of $\|Z\|_2$, and $U\in \Re^{m\times r}$ and
$V\in \Re^{n\times r}$ be formed by the leading left and right singular vectors of $Z$. Then, $(X,Y) = (X,  M-X)$, neither is zero, is optimal   for \eqref{RPCA} if and only if
$X = U T V^\top $ for some  $T \succeq 0 $ and $ Z/(\gamma \|Z\|_2) \in \partial \| M - X\|_1$.
\end{theorem}
\begin{proof}
The fact that $M\neq 0$ implies that the optimal value of  \eqref{RPCA}, denoted by $p^*$, is positive and finite.
By strong duality, the dual problem has optimal value $1/p^*$, which is also positive and finite.

%

%
Suppose that $X$ is constructed from an arbitrary optimal solution $Z$ of \eqref{RPCA-GD} such that
\begin{equation}\label{XZ-conditions}
  X = U T V^\top  \text{~~for some~~} T\succeq 0   \text{~ and ~} Z/(\gamma \|Z\|_2) \in \partial \| M - X\|_1.
\end{equation}
We next show that $(X,Y) = (X, M-X)$ and $Z$ satisfy the conditions 1-6 given in Theorem \ref{thm:optimality-conditions}, and thus $(X, M-X)$
must be an optimal solution for \eqref{RPCA}.
By definition, conditions 1, 4, 5 are satisfied (1 is trivial; the complementarity condition 4 is true because of the form $UTV^\top $ we have assigned to $X$; Condition 5 is due to the symmetry and positive semidefiniteness  of $T$. Note that, the linear spaces identified by $U$ and $V$ (or by $\sigma_{\max}(Z)$) are unique. However, in the SVD representation of $Z$, $U$ can be replaced by $UQ$, where $Q$ is any orthonormal matrix of  size $r\times r$, provided that  $V$ is replaced by $VQ$ simultaneously).
Condition 2 also holds because the objective of \eqref{RPCA-GD} is positively homogeneous and the dual optimal value, $1/p^*$, is positive and finite.
Furthermore, by letting $\|\cdot\|_\diamond = \|\cdot\|_\infty$ and  $\|\cdot\|_* = \|\cdot\|_1$ in Lemma \ref{thm:dual-norm}, and note that by our assumption $M - X\neq 0$, the condition
$Z/(\gamma \|Z\|_2) \in \partial \|M-X\|_1$ is equivalent to
\[
\left\langle \frac{Z}{\gamma \|Z\|_2}, M-X\right\rangle = \|M-X\|_1  \text{~~and~~} \left\|\frac{Z}{\gamma \|Z\|_2}\right\|_{\infty} = 1,
\]
which implies conditions 3 and 6. 

%
Conversely, suppose that $(X,Y) = (X, M - X)$, neither is zero,  is an optimal solution for \eqref{RPCA}, we need to show that $X$ and $Z$ are linked by
\eqref{XZ-conditions}.
In this case, we have conditions 1-6 of Theorem \ref{thm:optimality-conditions} on hand. Condition 4 implies that any left (resp.  right) singular vector of $Z$ associated to $\sigma_i(Z)$ with $i > r$ must be in the null space of $X^\top $ (resp. $X$). Therefore, $X$ has the form $ X = U T V^\top  $ for  a certain real matrix $T$ of  size $r\times r$.
Condition 5, which says that $X$ and $Z$ must share a simultaneous ordered SVD, ensures that  $T$ is also symmetric and positive semidefinite.
On the other hand,  condition  6 implies that $\|Z/(\gamma \|Z\|_2)\|_{\infty} = 1$. As a result, by further considering Lemma \ref{thm:dual-norm}, $Z/(\gamma \|Z\|_2) \in \partial \| M - X\|_1$ follows from condition 6. This completes the proof.
\end{proof}

\subsection{The whole gauge dual approach for robust PCA}
According to Theorem \ref{thm:primal-recovery}, a nontrivial primal optimal solution $(X,Y)$ for \eqref{RPCA} can be recovered from a given dual optimal solution $Z$ for \eqref{RPCA-GD} via solving the following $\ell_1$-minimization problem
\begin{eqnarray}\label{T-sub}
T^* \leftarrow \arg\min_{T \succeq 0} \left\langle  -Z/ \|Z\|_2, M-U T V^\top  \right\rangle + \gamma \| M - U T V^\top \|_1.
\end{eqnarray}
Note that the objective function in \eqref{T-sub} is guaranteed to be bounded below because
\begin{eqnarray*}
 \inf_{T\succeq 0} \left\langle  -Z/ \|Z\|_2, M-U T V^\top  \right\rangle + \gamma \| M - U T V^\top \|_1
 \geq  \inf_{Y\in \Re^{m\times n}} \left\langle -Z/ \|Z\|_2, Y \right\rangle + \gamma \|Y\|_1
 \geq  0.
\end{eqnarray*}
Here the second ``$\geq$" follows from the assumption that $\|Z/ (\gamma \|Z\|_2) \|_\infty = 1$, because  there is no nontrivial primal optimal solution otherwise.
The problem \eqref{T-sub} can be solved, e.g., by the classical ADMM,  which is well known to a broad audience and has a very fruitful literature, see, e.g., \cite{SWZ14}.
Here, we only present the iterative formulas but omit the details. Define $\hat Z := -Z/ \|Z\|_2$. To solve \eqref{T-sub} by ADMM, one iterates for $k\geq 0$ from a starting point $(T^0, W^0)$ as follows:
   \begin{subequations}
\begin{eqnarray}
\label{T-Y}
   Y^{k+1} &=& \arg\min_Y   \gamma \| Y\|_1  + \frac{\beta}{2}  \|Y + U T^k V^\top  - M - W^k/\beta + \hat Z/\beta\|_F^2, \\
\label{T-T} T^{k+1} &=&  \text{Proj}_{\succeq 0} \left(U^\top (M  +  W^k /\beta -   Y^{k+1})V\right), \\
\label{T-W} W^{k+1} &=& W^k - \beta (Y^{k+1} + U T^{k+1} V^\top  - M).
\end{eqnarray}
\end{subequations}
Here $\beta>0$ is a penalty parameter, $W^k\in\Re^{m\times n}$ is the Lagrange multiplier, and $\text{Proj}_{\succeq 0}(\cdot)$ denotes the Euclidean projection onto the symmetric positive semidefinite cone. Note that both \eqref{T-Y} and \eqref{T-T} have closed form solutions, and \eqref{T-T}  only involves  eigenvalue decomposition of a small matrix.

Now, we summarize below the steps for solving the robust PCA problem \eqref{RPCA} via gauge duality.
\begin{enumerate}
  \item Solve the gauge dual problem \eqref{RPCA-GD} to obtain an optimal solution $Z$;
  \item Compute the leading  singular vectors of $Z$:  $U\in\Re^{m\times r}$  and $V\in\Re^{n\times r}$;
  \item Obtain $T^*$ from solving \eqref{T-sub};
  \item Set $X = U T^* V^\top $ and $Y = M - X$.
\end{enumerate}

In principle, the gauge dual problem \eqref{RPCA-GD} can be solved via any numerical  algorithms using only subgradients, e.g., projected subgradient method and
bundle type methods \cite{Sho85,LNN95,Kiw95,BN05}. Indeed, according to basic subgradient calculus \cite{Roc70}, the subdifferential of $ \max \{ \|Z\|_2, \|Z\|_\infty/\gamma\}$ can be characterized easily. See also, e.g., \cite{Lew95}.
%
In particular, the subdifferential of $\|Z\|_2$ is given by
\[
\partial \|Z\|_2 = \{  U \diag{\mu}V^\top  \ | \ \mu \in \partial\|\sigma(Z)\|_\infty, Z = U \diag{\sigma(Z)}V^\top , U\in{\cal U}_m, V\in{\cal U}_n\},
\]
and that of $\|Z\|_\infty$ is given by
$\partial \|Z\|_\infty = \text{conv}\{ \text{sgn}(Z_{ij})E_{ij}   \ | \ |Z_{ij}| = \|Z\|_{\infty} \}$, where ``conv" denotes the convex hull operation.
As a result, the subdifferential of $ \max \{ \|Z\|_2, \|Z\|_\infty/\gamma\}$ is given by
\[
\partial \max \{ \|Z\|_2, \|Z\|_\infty/\gamma\} =
\left\{
  \begin{array}{ll}
     \partial \|Z\|_2, & \hbox{if $\|Z\|_2 > \|Z\|_\infty/\gamma$;} \\
     \text{conv}\{\partial \|Z\|_2 \cup \partial \|Z\|_\infty/\gamma\}, & \hbox{if $\|Z\|_2 = \|Z\|_\infty/\gamma$;} \\
     \partial \|Z\|_\infty/\gamma, & \hbox{if otherwise.}
  \end{array}
\right.
\]
It is worth noting that in practice the computation of one subgradient at any given $Z$ is very easy,
because in this case only one pair of left and right singular vectors corresponding to the largest
singular value needs to be computed. This fact is actually the  key to avoid full SVD.
On the other hand, it is apparent that the projection of any given $X$ onto $\{Z \ | \ \langle M, Z\rangle \geq 1\}$ is very easy and is given by
\begin{equation*}\label{Proj}
{\cal P}(X) =
\left\{
  \begin{array}{ll}
    X, & \hbox{if $\langle M, X \rangle \geq 1$;} \\
    X + \frac{1 - \langle M, X \rangle}{\langle M, M \rangle} M, & \hbox{if otherwise.}
  \end{array}
\right.
\end{equation*}
%

\section{Application to SDP}\label{sec:sdp}
In this section, we consider the SDP problem in standard form and derive a gauge duality approach. Let  $S^n$ be
 the set of real symmetric matrices of order $n$.
Given $C\in S^n$, $b\in\Re^m$, and a linear operator ${\cal A}:  S^n \rightarrow \Re^m$,
we consider the SDP problem in standard form
\begin{eqnarray}\label{SDP-Standard}
 \min\nolimits_X  \{ \inner{C}{X} \ | \ {\cal A} X = b, \ X\succeq 0\}.
\end{eqnarray}
Most algorithms, e.g., interior point methods and ADMM \cite{WGY10,XuYY11}, for solving \eqref{SDP-Standard} require full eigenvalue decomposition at each iteration.
In a recent work \cite{Ren14},  under certain assumption that does not loss generality,
Renegar reformulated \eqref{SDP-Standard} as
\begin{equation}\label{Renegar-SDP}
\max\nolimits_X \{ \lambda_{\min}(X) \ | \ {\cal A}X = b, \langle C, X\rangle = z\},
\end{equation}
where $\lambda_{\min}(\cdot)$ denotes the smallest eigenvalue, and
$z$ is any value strictly less than the optimal value of \eqref{SDP-Standard}.
After solving the nonsmooth problem \eqref{Renegar-SDP}, a solution of \eqref{SDP-Standard} can be obtained directly via a simple linear transform.
Indeed,  full eigenvalue decomposition can be avoided if one solves \eqref{Renegar-SDP} by subgradient type methods, e.g., the projected subgradient method.
We note that  the projection on the feasible set of \eqref{Renegar-SDP} could be costly.

The Lagrange dual problem of \eqref{SDP-Standard} is
\begin{eqnarray}\label{SDP-LD}
 \max\nolimits_{y}  \{ b^\top y  \ | \ C - {\cal A}^\top  y  \succeq 0\}.
\end{eqnarray}
Throughout this section, we make the following  blanket  assumption on the primal and dual Lagrange pair \eqref{SDP-Standard} and \eqref{SDP-LD}, which will not be mentioned further.
\begin{assumption}\label{assumption-sdp}
Assume that both \eqref{SDP-Standard} and \eqref{SDP-LD} have optimal solutions. Furthermore,   assume, without loss of generality, that $C\succeq 0$, because, otherwise, one can replace $C$ by $C - {\cal A}^\top  \hat y$ for some $\hat y$ that is  feasible for \eqref{SDP-LD}. 
\end{assumption}

\subsection{The gauge dual problem}
Under Assumption \ref{assumption-sdp}, the objective  value of \eqref{SDP-Standard} at any $X\succeq 0$ is nonnegative.
Let $\delta(X \ | \ \cdot \succeq 0)$ be the indicator function of the set $\{X \in S^n \ | \ X\succeq 0\}$, i.e., $\delta(X \ | \ \cdot \succeq 0)$ equals $0$ if $X\succeq 0$ and $+\infty$ if otherwise.
Define $\kappa(X) := \inner{C}{X} + \delta(X \ | \ \cdot \succeq 0)$, which is clearly  a gauge function. Then, the primal SDP \eqref{SDP-Standard} is equivalently transformed to the following gauge optimization problem
\begin{eqnarray}\label{SDP-GP}
 \min\nolimits_X  \{ \kappa(X) \ | \ {\cal A} X = b \}.
\end{eqnarray}
This kind of gauge optimization formulation has been considered first in \cite{Fre87} for linear programming and then in \cite{FMP14} for general conic programming.
The   gauge dual problem of \eqref{SDP-GP}  follows directly from \cite[Sec. 7.2]{FMP14} and is summarized below.
\begin{proposition}
The gauge dual problem of \eqref{SDP-GP} is given by
\begin{eqnarray}\label{SDP-GD}
 \min\nolimits_y  \{ \kappa^\circ ({\cal A}^\top  y) \ | \ b^\top y \geq 1\},
\end{eqnarray}
where $\kappa^\circ (Z) = \inf_{\mu}  \{\mu \geq 0 \ | \  \mu {C} - Z  \succeq 0   \} $, $Z\in S^n$.
\end{proposition}

In the following, we use the notation $\mu_y := \kappa^\circ({\cal A}^\top y)$  for   $y\in\Re^m$.

\subsection{Optimality conditions}
To derive necessary and sufficient conditions for $X^*$ and $y^*$ to be primal and dual optimal, respectively,
we need   strong duality  between   \eqref{SDP-GP} and \eqref{SDP-GD}, that is
   \begin{eqnarray}
     \label{StrongDuality-SDP}
  [ \min_X  \{ \kappa(X) \ | \ {\cal A} X = b \} ] \cdot
  [ \min_y  \{ \kappa^\circ ({\cal A}^\top  y) \ | \ b^\top y \geq 1\} ] = 1.
   \end{eqnarray}
We note that  the strong duality condition  \eqref{StrongDuality-SDP}  may fail to hold in general. The following theorem gives  a sufficient condition to ensure its validity.
\begin{theorem}
  [Strong duality]
  Assume that $C\succ 0$, $b\neq 0$, and there exists $X\succ 0$ such that ${\cal A}X = b$.
  Then, both  \eqref{SDP-GP} and \eqref{SDP-GD} achieve their respective optimum value and   strong duality condition \eqref{StrongDuality-SDP} holds.
\end{theorem}
 \begin{proof}
 Let ${\cal C} := \{X \ | \ {\cal A} X = b\}$. It is easy to show that ${\cal C}' := \{{\cal A}^\top y \ | \ b^\top y \geq 1\}$.
   We prove the theorem by showing that the conditions of \cite[Corollary 5.4]{FMP14} are satisfied. First, it is clear that
   \[
   \mathrm{ri}({\cal C}) \cap   \mathrm{ri}(\mathrm{dom}(\kappa)) = {\cal C}  \cap  \{X \ | \ X \succ 0\},
   \]
   which is nonempty by our assumption.  Second, the assumptions that $C \succ 0$ and $b\neq 0$ imply, respectively, that $\kappa^\circ $ has full domain and ${\cal C}'$ is nonempty.
   Furthermore, ${\cal C}'$ is clearly convex. Since a nonempty convex set always has nonempty relative interior, it is then
   clear that the intersection $\mathrm{ri}({\cal C}') \cap   \mathrm{ri}(\mathrm{dom}(\kappa^\circ))$ is also nonempty.
   It is also apparent that $\kappa$ is a closed gauge function.    As a result, the conclusion of this theorem follows directly from
   \cite[Corollary 5.4]{FMP14}.
 \end{proof}
\begin{theorem}
  [Optimality conditions]\label{thm:opt-sdp}
  Assume that  \eqref{SDP-GD} has an optimal solution,   $ \min_X  \{ \kappa(X) \ | \ {\cal A} X = b \} > 0$, and strong duality holds
  Then,
  $X^*\in S^n$ and $y^*\in \Re^m$ are, respectively, optimal for  \eqref{SDP-GP} and \eqref{SDP-GD} if and only if
  ${\cal A}X^* = b$, $X^*\succeq 0$, $b^\top y^* = 1$, and $\langle \mu_{y^*} C - {\cal A}^\top  y^*, X^*\rangle = 0$, where $\mu_{y^*} := \kappa^\circ ({\cal A}^\top y^*)$.
\end{theorem}
\begin{proof}
   Under our assumption, it is clear that $X^*$ and $y^*$ are, respectively, optimal for \eqref{SDP-GP} and \eqref{SDP-GD} if and only if
 ${\cal A}X^* = b$, $X^*\succeq 0$, $b^\top y^* \geq 1$, and strong duality holds. In this case, we have
 \begin{eqnarray*}
  1 = \kappa(X^*) \kappa^\circ({\cal A}^\top y^*)
   = \mu_{y^*} \langle  C, X^*\rangle
     \geq  \langle {\cal A}^\top y^*, X^*\rangle
     =  b^\top y^* \geq 1,
\end{eqnarray*}
where  the first ``$\geq$" follows form the definition of polar function. Thus, all inequalities become  equalities.  In particular, we have $b^\top y^*=1$ and
$\langle \mu_{y^*} C - {\cal A}^\top  y^*, X^*\rangle = 0$.
\end{proof}

We note that the condition $\langle \mu_{y^*} C - {\cal A}^\top  y^*, X^*\rangle = 0$ is a combination and generalization of the conditions 4 and 5 of \cite[Prop. 3.2]{FM15}, where $C$ is restricted to the identity matrix.

\subsection{Primal solution recovery}
After obtaining $y^*$ from solving the gauge dual problem \eqref{SDP-GD}, we need to recover an optimal solution $X^*$ of \eqref{SDP-GP}.
To validate a way of doing this, we first present a simple lemma from linear algebra.
\begin{lemma}
  \label{AB>=0}
  Let $A$ and $B$ be any real symmetric positive semidefinite matrices of order $n$. If $\langle A, B\rangle = 0$, then there exists an orthonormal matrix $U$ of order $n$ such that
  both $Q^\top A Q$ and $Q^\top B Q$ are diagonal.
\end{lemma}

The following theorem validates a way of recovering an optimal solution $X^*$ of \eqref{SDP-GP} from that of \eqref{SDP-GD}, and
its proof relies on the optimality conditions established in Theorem \ref{thm:opt-sdp} and Lemma \ref{AB>=0}. In comparison, the proof of the corresponding results in \cite{FM15} relies on the von Neumann's trace inequality. Though the proof here is very simple, we present it for completeness.

\begin{theorem}
  [Primal solution recovery]
  Let $y^*\in\Re^m$ be any optimal solution of the gauge dual problem \eqref{SDP-GD}.
  If $X^*\in S^n$ is optimal for \eqref{SDP-GP}, then  ${\cal A} X^* = b$ and
$X^* = U_2 T U_2^\top $ for some $T\succeq 0$.
Here, $U_2 \in \Re^{n\times r}$ is an orthonormal basis of the null space of $\mu_{y^*} C - {\cal A}^\top  y^*$.
Therefore, $X^*$ can be recovered from $y^*$ via $X^* = U_2 T^* U_2^\top $ with
$T^* \in \arg\min_{T \succeq 0} \|{\cal A} U_2 T U_2^\top  - b \|_2$.
\end{theorem}
\begin{proof}
   It follows from Theorem \ref{thm:opt-sdp} that $\langle \mu_{y^*} C - {\cal A}^\top  y^*, X^*\rangle = 0$.
   Since both  $\mu_{y^*} C - {\cal A}^\top  y^*$ and  $X^*$ are also symmetric and positive semidefinite, it follows from Lemma \ref{AB>=0} that they can be simultaneously diagonalized by an orthonormal matrix.
   Suppose that $U = (U_1,U_2)$ is an orthonormal matrix such that
\[   \mu_{y^*} C - {\cal A}^\top  y^* = U \left(
                                       \begin{array}{cc}
                                         \text{diag}(\lambda^+(\mu_{y^*} C - {\cal A}^\top  y^*)) & 0 \\
                                         0 & 0 \\
                                       \end{array}
                                     \right) U^\top ,
                                     \]
                                     where $\lambda^+(\mu_{y^*} C - {\cal A}^\top  y^*)$ contains the strictly positive eigenvalues of $\mu_{y^*} C - {\cal A}^\top  y^*$.
Then, $X^*$  has the form
$X^* = U_2 T U_2^\top $ for some $ T\succeq 0$, where the size of $T$ is equal to the dimension of the null space of $\mu_{y^*} C - {\cal A}^\top  y^*$.
The remaining claim of this theorem follows directly.
\end{proof}

We note that, when restricted to \eqref{Fried-trace} or \eqref{Fried-Cpd} with $\varepsilon=0$ and $C$ being positive definite, the primal solution recovery result presented here reduces to \cite[Corollary 6.1, Corollary 6.2]{FM15}.

\subsection{Computing a subgradient of $\kappa^\circ \circ {\cal A}^\top $}

Now, we discuss  how to compute a subgradient of $\kappa^\circ \circ {\cal A}^\top $ at any given $y$.
Let $y\in\Re^m$ and $Z_y\succeq 0$ be such that $\langle C, Z_y\rangle \leq 1$ and $\mu_y = \kappa^\circ ({\cal A}^\top y) = \langle {\cal A}^\top y, Z_y\rangle$, that is,
the supremum defining $\kappa^\circ ({\cal A}^\top y)$ is attained at $Z_y$. Then, it is easy to show that
${\cal A}Z_y \in \partial (\kappa^\circ \circ {\cal A}^\top)(y) $. Thus, the remaining issue is  how to compute $Z_y$ from the given $y$.
It is clear that  $\mu_y = 0$ if ${\cal A}^\top  y\preceq 0$, and in this case the supremum  defining $\kappa^\circ ({\cal A}^\top  y)$ is attained at $Z_y = 0$.
On the other hand, if $b^\top y\geq 1$ and $\mu {C} - {\cal A}^\top  y \not\succeq 0 $ for any finite $\mu>0$, then $\mu_y = +\infty$, in which case $y$ is essentially infeasible for \eqref{SDP-GD} and thus $Z_y$ does not exist.
Now, consider the most common case, i.e.,  $\mu_y \in(0, +\infty)$.
In this case, it is clear that the evaluation of $\mu_y$ reduces to computing the smallest generalized eigenvalue of the pencil $({\cal A}^\top y, C)$. In principle, any numerical methods relying on merely matrix-vector multiplications for such problem can be applied.
For example, $\mu_y$  can be obtained by the inverse free preconditioned Krylov subspace method \cite{GY02}, which has a black-box implementation in MATLAB named
{\tt eigifp}\footnote{Available at: \url{http://www.ms.uky.edu/~qye/software.html}} that applies to  symmetric and positive definite $C$.
%
The fact that the evaluation of $\mu_y$ can rely on numerical methods using only matrix-vector multiplications is the key to avoid  full eigenvalue decomposition  in the whole gauge dual approach for solving SDP.
%
%
%
Suppose that $\mu_y {C} - {\cal A}^\top  y$ has the following eigenvalue decomposition
\[
\mu_y {C} - {\cal A}^\top  y  = Q \Lambda_y Q^\top
= (Q_{\cal I}, Q_{\cal J}) \left(
           \begin{array}{cc}
             \text{diag}(\lambda^+(\mu_y {C} - {\cal A}^\top  y)) &   \\
               & 0 \\
           \end{array}
         \right) (Q_{\cal I},Q_{\cal J})^\top ,
\]
where $\lambda^+(\mu_y {C} - {\cal A}^\top  y)$ contains the strictly positive eigenvalues of $\mu_y {C} - {\cal A}^\top  y$,
and ${\cal I}$  and ${\cal J}$ are, respectively, index sets corresponding to the strictly positive and zero eigenvalues of $\mu_y {C} - {\cal A}^\top  y$.
Then, $Z_y = Q_{\cal J} \text{diag}(\lambda^+(Z_y)) Q_{\cal J}^\top $, where $\lambda^+(Z_y) \in \Re^{|{\cal J}|}_+$ and can determined via
\begin{eqnarray*}
\mu_y
&=&  \kappa^\circ ({\cal A}^\top y) = \inner{{\cal A}^\top  y}{Z_y} \\
&=&  \inner{{\cal A}^\top  y}{Q_{\cal J} \text{diag}(\lambda^+(Z_y)) Q_{\cal J}^\top } \\
&=&  \inner{Q_{\cal J}^\top  ({\cal A}^\top  y) Q_{\cal J}}{ \text{diag}(\lambda^+(Z_y)) } \\
&=&  \inner{\text{diag}\left(Q_{\cal J}^\top  ({\cal A}^\top  y) Q_{\cal J}\right)}{ \lambda^+(Z_y)  }.
\end{eqnarray*}
Clearly, $Z_y$ thus defined satisfies  $\inner{C}{Z_y} =  1$, $Z_y \succeq 0$, and  $\mu_y = \kappa^\circ ({\cal A}^\top y) = \inner{{\cal A}^\top  y}{Z_y}$.
As a result,  the supremum defining $\kappa^\circ ({\cal A}^\top y)$ is attained at $Z_y$. We note that intuitively   ${\cal J}$ is a small set. Thus, $Z_y$ is low rank.
In practice, one   needs to compute only one   normal vector in the null space of $\mu_y C - {\cal A}^\top y$, which is sufficient to determine a $Z_y$ and thus an element of
$\partial (\kappa^\circ \circ {\cal A}^\top) (y)$.


\subsection{The whole gauge dual approach for \eqref{SDP-GP}}
Clearly, the projection on the feasible set of \eqref{SDP-GD} is very easy. This, together with the fact that the  computation of a subgradient of  $\kappa^\circ \circ {\cal A}^\top $ at any given $y$ does not require full eigenvalue decomposition, provides the potential to design efficient algorithms for solving SDP.
We summarize below the steps for solving \eqref{SDP-GP} via gauge duality.
\begin{enumerate}
  \item Solve the gauge dual problem \eqref{SDP-GD} to obtain an optimal solution $y^*$;
  \item Compute $U_2 \in \Re^{n\times r}$,  an orthonormal basis of the null space of $\mu_{y^*} C - {\cal A}^\top  y^*$;
  \item Obtain $T^*$ from solving $\min_{T \succeq 0} \|{\cal A} U_2 T U_2^\top  - b \|_2$ and set    $X^* = U_2 T^* U_2^\top $.
\end{enumerate}

For the trace minimization problem considered in \cite{FM15}, the coefficient matrix $C$ is the identity matrix. In this case, the objective function of 
\eqref{SDP-GD} reduces to $\kappa^\circ({\cal A}^\top y) = \max(\lambda_{\max}({\cal A}^\top y), 0)$, and the computation of a subgradient of $\kappa^\circ({\cal A}^\top y)$ at any given $y$ requires only the computation of an eigenvector corresponding to the largest eigenvalue of ${\cal A}^\top y$, which is simpler than our case with a general $C\succeq 0$. Furthermore, in this case $U_2$ in Step 2 is no more than an orthonormal basis of the vector space spanned by the eigenvectors corresponding to the largest eigenvalues of $ {\cal A}^\top  y^*$. As a result, the above scheme is a generalization of that studied in \cite{FM15} from $C$ being the identity matrix to general $C\succeq 0$.

\section{Concluding remarks}\label{sec:concluding}
In this paper, we studied gauge duality theory for robust PCA and SDP problems.
Unlike the classical Lagrange dual, both the primal and the dual problems are minimization problems in gauge dual. Furthermore, the variable of the gauge dual problem is not a multiplier attendant to the primal constraints. As a consequence, a primal recovery procedure, which recovers a primal optimal solution from a dual one, needs to be designed.
Our main contribution of this paper is to validate such a way for both robust PCA and SDP via analyzing their optimality conditions.
%
%
A few  important issues need to be explored further, e.g.,
developing for nonsmooth optimization fast and efficient numerical algorithms that rely merely on subgradients, the interplay and acceleration between primal and dual intermediate inexact solutions, the influence of inexact dual solution on  primal solution recovery. It is hoped that this work could serve as a starting point for the design and analysis  of  gauge duality based numerical algorithms.

\section*{Acknowledgement}
We thank three anonymous referees for their thoughtful and insightful comments, which improved the paper greatly.


\def\cprime{$'$}

\end{document}